\documentclass[reqno,12pt]{amsart}
\usepackage{amssymb,amsfonts,graphicx}
\usepackage{geometry}

\usepackage{mathrsfs}
\usepackage{float}

\usepackage[toc,page,title,titletoc,header]{appendix}
\usepackage{epsfig}
\usepackage{mathrsfs}
\usepackage{float}

\usepackage{amssymb}
\usepackage{mathrsfs}
\usepackage{float}
\usepackage{amsmath}
\usepackage{amsfonts,amssymb}
\usepackage{graphicx,color}
\usepackage{bm}

\usepackage{fancyhdr}
\allowdisplaybreaks

\begin{document}
\setlength{\parindent}{1.5em}
\theoremstyle{plain}
\newtheorem{theorem}{Theorem}[section]
\newtheorem{definition}{Definition}[section]
\newtheorem{lemma}{Lemma}[section]
\newtheorem{Example}{Example}[section]
\newtheorem{cor}{Corollary}[section]
\newtheorem{remark}{Remark}[section]
\newtheorem{property}{Property}[section]
\newtheorem{proposition}{Proposition}[section]


\title{Formation and propagation of singularities in one-dimensional Chaplygin gas}
\maketitle
 \centerline{De-Xing Kong$^{a} $\quad and\quad Changhua Wei$^{b,*}$}
\begin{center}
$^*\!${\it {\small Corresponding author }}\\
$^{a,b}\!${\it {\small  Department of Mathematics, Zhejiang University, Hangzhou 310027, China}}\\
$^{b}\!$
{\it  {\small E-mail address: changhuawei1986@gmail.com}}\\

\end{center}

        \begin{abstract}
        In this paper, we investigate the formation and propagation of singularities for the system for one-dimensional Chaplygin
        gas, which is described by a quasilinear hyperbolic system with linearly degenerate characteristic fields. The phenomena of concentration and the formation of ``$\delta$-shock'' waves are identified and analyzed systematically for this system under suitably large initial data. In contrast to the Rankine-Hogoniot conditions for classical shock, the generalized Rankine-Hogoniot conditions for ``$\delta$-shock'' waves are established. Finally, it is shown that
         the total mass and momentum related to the solution are independent of time.

        \vskip 3mm

        \end{abstract}

\baselineskip=7mm

  \section{Introduction}
As we know, smooth solutions of nonlinear hyperbolic
systems generally exist in finite time even if initial data is
sufficiently smooth and small. After this time, only weak solutions
can be defined. Therefore, the following questions arise naturally: what kinds of singularities will appear at the blowup
point and how do the singularities propagate?

Let us first recall some classical results on this research topic. Consider the Cauchy problem for conservation laws in one dimensional space
$$
\left\{\begin{array}{l}
u_{t}+F(u)_{x}=0 \qquad \text{in}\; \mathbb R\times[0,\infty),\vspace{2mm}\\
u(0,x)=u_{0}(x)\qquad \text{in}\;\mathbb R\times\{t=0\},
\end{array}
\right.
\eqno(1.1)
$$
where $u=(u_{1}(t,x),\cdots,u_{m}(t,x))$ is the unknown vector-valued function, $F:\mathbb R^{m}\rightarrow\mathbb R^{m}$ is a given flux function, $u_{0}:\mathbb R\rightarrow\mathbb R^{m}$
is the initial data which is a given smooth vector-valued function.

 It is well-known that one of classical blowup phenomena for system (1.1) is the formation of shock. The formation of shock for one or several space dimensions of hyperbolic conservation laws is an important research topic, which has attracted many mathematicians and physicists widespread attention (see \cite{A,1,Ch1,Ch2,2,6,9,11,LT,13}). As pointed out by Majda in \cite{M}, singularity of shock is due to the catastrophe of the first order derivatives of system (1.1), while the solution itself is bounded in the domain considered. Thus, the weak solution of the Cauchy problem (1.1) can be defined as follows (see \cite{14}).
\begin{definition}
A bounded measurable function $u=u(t,x)$ is called a weak solution of the Cauchy problem (1.1) with bounded and measurable initial data $u_{0}$, if it holds that
$$
\iint\limits_{t\geq0} (u\phi_{t}+F(u)\phi_{x}) dxdt+\int_{t=0}u_{0}\phi dx=0
\eqno(1.2)
$$
for all $\phi\in C_{0}^{1}$, where $C_{0}^{1}$ is the class of $C^{1}$ functions, which vanish outside of a compact subset.
\end{definition}

In this paper, we shall consider the Cauchy problem for the system for one-dimensional Chaplygin gas
$$
\left\{\begin{array}{l}
{\displaystyle \partial_t\rho+\partial_x(\rho u)=0 },\vspace{2mm}\\
{\displaystyle \partial_t (\rho u)+\partial_x(\rho u^2+p)=0}, \vspace{2mm}\\
\end{array}\right.
\eqno(1.3)
$$
where $ \rho=\rho(t, x)$ and
  $u=u(t, x)$ denote the density and the velocity, respectively, and
$ p(t, x)$ is the pressure which is a function of $ \rho$ given by
$$
p=p_0-\frac{\mu^2}{\rho},
\eqno(1.4)
$$
in which $ p_0 $ and $ \mu $ are two positive constants.

It is easy to check that the system (1.3) is a fully linearly degenerate system. The system (1.3) with the state equation (1.4) was introduced by Chaplygin \cite{C}, Tsian \cite{T} and von K\'{a}rm\'{a}n \cite{K} as a suitable mathematical approximation for calculating the lifting force on a wing of an airplane in aerodynamics. Brenier \cite{B} studied the one dimensional Riemann problems and constructed the solutions with concentration phenomena.

 In this paper, we will illustrate that the formation of singularity is due to the overlapping of linearly degenerate characteristics and the density $\rho$ tends to infinity at the blowup points. Thus, there are ``nonclassical'' solutions in this sense, in contrast to the above classical results, the Cauchy problem for this system of conservation laws does not possess a weak $L^{\infty}$-solution. In order to solve this Cauchy problem in the framework of nonclassical solutions, it is necessary to construct the solution for this Cauchy problem with ``strong singularities''. Fortunately, in the past two decades, $\delta$-shock wave was constructed to describe this phenomenon (see \cite{B,Chen,DS,HF,Li,N,S,Wang,Yang}). Roughly speaking, $\delta$-shock wave is a kind of discontinuity, on which at least one of the state variables may develop an extreme concentration in the form of a weighted Dirac delta function with the discontinuity as its support.

 For completeness, we recall the definition of $\delta$-shock (see \cite{Chen, Yang}).
\begin{definition}
A two-dimensional weighted delta function $w(s)\delta_{\mathbb{S}}$ supported on a smooth curve $\mathbb{S}$ parameterized as $(t(s),\,x(s))\,\,(a\leq s\leq b)$ is defined by
$$
<w(s)\delta_{\mathbb{S}},\phi(t,x)>:=\int^{b}_{a}w(s)\phi(t(s),x(s))\sqrt{(t^{'}(s))^{2}+(x^{'}(s))^{2}}ds
\eqno(1.5)
$$
for all $\phi\in C_{0}^{\infty}$ and $^{'}$ denotes $\frac{d}{ds}$.
\end{definition}
With the above definition, a family of $\delta$-measure solutions of system (1.3) with parameter $t$ can be obtained as
$$
\left\{
\begin{array}{l}
\rho(t,x)=\hat{\rho}(t,x)+w(t)\delta_{\mathbb{S}},\vspace{2mm}\\
u(t,x)=\hat{u}(t,x),
\end{array}\right.
\eqno(1.6)
$$
where the discontinuity
$$\mathbb S=\{(t,x(t)):0\leq t< \infty\}$$ and
$$
\hat{\rho}(x,t)=\rho_{-}+[\rho]\chi(x-x(t)),\quad \hat{u}(t,x)=u_{-}+[u]\chi(x-x(t)),
\eqno(1.7)
$$
in which $\rho_{\pm},u_{\pm}\in L^{\infty}(\mathbb R \times [0,\infty))\cap C^{1}(\mathbb R\times [0,\infty))$, $h:=h_{+}-h_{-}$ denotes the jump of function $h$ across the discontinuity and $\chi(x)$ is the characteristic function that is 0 when $x<0$ and 1 when $x>0$.
\begin{definition}
A pair $(\rho,u)=(\rho(t,x),u(t,x))$ is called a $\delta$-measure solution of system (1.3) in the sense of distributions if there exist a smooth curve $\mathbb S$ and a weighted $w$ such that $\rho$ and $u$, which are defined by (1.6), satisfy
$$
<\rho,\phi_{t}>+<\rho u,\phi_{x}>=0
\eqno(1.8)
$$
and
$$
<\rho u,\phi_{t}>+<\rho u^{2}+p(\rho),\phi_{x}>=0
\eqno(1.9)
$$
for all $\phi\in C^{\infty}_{0}$, where
$$
<\rho,\phi>=\int_{0}^{\infty}\int_{-\infty}^{\infty}\hat{\rho}\phi dxdt+<w\delta_{\mathbb{S}},\phi>,
$$
$$
<\rho u,\phi>=\int_{0}^{\infty}\int_{-\infty}^{\infty}\hat{\rho }\hat{u}\phi dxdt+<w u_{\delta}\delta_{\mathbb{S}},\phi>
$$
and
$$
<\rho u^{2}+p(\rho),\phi>=\int_{0}^{\infty}\int_{-\infty}^{\infty}(\hat{\rho}\hat{u}^{2}+p(\hat{\rho}))\phi dxdt+<wu_{\delta}^{2}w_{\mathbb{S}},\phi>,
$$
where $u_{\delta}=\frac{dx(t)}{dt}$ denotes the propagation speed of the discontinuity.
\end{definition}
\begin{remark}
According to \cite{B}, the pressure $p$ is a nonlinear term of $\rho$ and it should be noticed that the delta measure does not contribute.
\end{remark}

 In fact, this paper continues our recent work \cite{KW}. In \cite{KW}, we systematically analyzed the singularity formed by the overlapping of characteristics (or the degeneracy of strictly hyperbolicity), where we named Delta-like singularity. Moreover, we presented a clear behavior in the neighborhood of the blowup point. Based on this, we further constructed the Delta-like singularity with point-shape and Delta-like singularity with line-shape by classical analysis instead of the measure solutions introduced here. However, for more interesting cusp-type singularity, we got nothing except the blowup behavior. In this paper, by the concept of $\delta$-shock wave type solution we establish the generalized Rankine-Hugonoit condition to describe the relationship among the location, propagation speed of the discontinuity, rate of change of its weights and reassignment of $\rho$ and $u$ on its discontinuity (see Theorem 3.1). In order to understand the physical meaning of $\delta$-shock wave solution, we show the total mass and momentum conservation under suitable assumptions on the state variable $(\rho,u)$ at the infinity (see Theorem 4.1).

The paper is organized as follows. In Section 2, some preliminaries on the linearly degenerate hyperbolic systems are given. For completion, we present some basic facts on the formation of cusp-type singularity, which have been investigated in \cite{KW}. Furthermore, we investigate the mechanism of formation of such a ``nonclassical'' singularity. The local existence and uniqueness of $\delta$-shock wave solution is constructed in Section 3. The physical meaning of $\delta$-shock wave solution is illustrated in Section 4. Section 5 gives some remarks and discussions.

\section{Preliminaries}
Consider the system (1.3) with the state equation (1.4), which describes the motion of
a perfect fluid characterized by the pressure-density relation (known as the Chaplygin or
von K\'{a}rm\'{a}n-Tsien pressure law). This endows the system a highly symmetric structure. This is evident if we adopt the local sound speed
$ c=c(t,x)=[p'(\rho)]^{1/2}$ and the usual mean velocity of the fluid $ u=u(t,x) $ as dependent variables. In this case, the system reads
$$
\partial_{t}U+A(U)\partial_{x}U=0,
\eqno(2.1)
$$
where
\begin{displaymath}
U=\left(\begin{array}{c}c\\u
\end{array}
\right)\quad \rm{and}\quad
A(U)=\left(\begin{array}{cc}u&-c\\-c&u
\end{array}
\right).
\end{displaymath}
Obviously, the eigenvalues of $A(U)$ read
$$
 \lambda_{-}=u-c,\quad\lambda_{+}=u+c.
 \eqno(2.2)
$$
Moreover, it is easy to verify that $\lambda_{\pm}=\lambda_{\pm}(t,x)$ are Riemann invariants.
Under the Riemann invariants, the system (1.3) can be reduced to
 $$\left\{\begin{array}{l}
{\displaystyle \partial_t\lambda_{-}+\lambda_{+}\partial_x\lambda_{-}=0 },\vspace{2mm}\\
{\displaystyle \partial_t\lambda_{+}+\lambda_{-}\partial_x\lambda_{+}=0}. \vspace{2mm}\\
\end{array}\right.
\eqno(2.3)
$$
Consider the Cauchy problem for the system (1.3) with the following initial data
$$
t=0: \rho=\rho_{0}(x),\quad u=u_{0}(x),
\eqno(2.4)
$$
where $ \rho_{0}(x)$  and $ u_{0}(x)$  are two suitably smooth functions with bounded $ C^{2} $ norm.
For consistency, let
$$
\lambda_{\pm}(0,x)=\Lambda_{\pm}(x)\triangleq u_{0}(x)\pm\frac{\mu}{\rho_{0}(x)}.
\eqno(2.5)
$$
Thus, studying the system (1.3) with initial data (2.4) is equivalent to studying the system (2.3) with initial data (2.5) in the existence domain of classical solutions.

In the existence domain of the classical solution of (2.3), we recall the definition of characteristics and denote two characteristics starting from $(0,\alpha)$ by
$$
x=x^{+}(t,\alpha),\quad x=x^{-}(t,\alpha),
$$
respectively, which satisfy
$$\left\{\begin{array}{l}
{\displaystyle \frac{dx^{+}(t,\alpha)}{dt}=\lambda_{-}(t,x^{+}(t,\alpha)) },\vspace{2mm}\\
{\displaystyle t=0:x^{+}(0,\alpha)=\alpha} \vspace{2mm}\\
\end{array}\right.
\eqno(2.6)
$$
and
$$\left\{\begin{array}{l}
{\displaystyle \frac{dx^{-}(t,\alpha)}{dt}=\lambda_{+}(t,x^{-}(t,\alpha))},\vspace{2mm}\\
{\displaystyle t=0:x^{-}(0,\alpha)=\alpha}, \vspace{2mm}\\
\end{array}\right.
\eqno(2.7)
$$ respectively. On the other hand, for any $(t,x)$ in the maximal domain of definition of a smooth solution , we define $\beta(t,x)$ by $\beta(t,x)=x^{+}(0;t,x)$ where $x^{+}(s;t,x)$ is the unique solution of the ODE
$$
\frac{dx^{+}(s;t,x)}{ds}=\lambda_{-}(s,x^{+}(s;t,x))
$$
with initial condition $x^{+}(t;t,x)=x$. $\alpha(t,x)$ is defined similarly.
\begin{lemma}By (2.3),
$\lambda_{+}(t,x)$ is constant along the curve $x=x^{+}(t,\alpha)$, while  $\lambda_{-}(t,x)$ is constant along the curve $x=x^{-}(t,\alpha)$.
\end{lemma}
The following lemma can be found in Kong-Zhang \cite{11}
\begin{lemma}
In terms of characteristic parameters $(\alpha,\beta)$ introduced above, it holds that
$$
t(\alpha,\beta)=\int_\alpha^\beta\frac{1}{\Lambda_{+}(\zeta)-\Lambda_{-}(\zeta)}d\zeta,
\eqno(2.8)
$$
$$
x(\alpha,\beta)=\frac{1}{2}\left\{{\alpha+\beta+\int_\alpha^\beta\frac{\Lambda_{+}(\zeta)+\Lambda_{-}(\zeta)}{\Lambda_{+}(\zeta)-\Lambda_{-}(\zeta)}}d\zeta\right\}.
\eqno(2.9)
$$
\end{lemma}

The following lemma, which is due to \cite{10}, plays an important role in our discussion.
 \begin{lemma}
  Adopt the above notations, if there exists $\alpha$ such that $\Lambda_{-}(\alpha)\neq\lambda_{+}(t,x^{-}(t,\alpha))$ for $t\geq0$, then it holds that
 $$
 \frac{\partial x^{-}(t,\alpha)}{\partial\alpha}=\frac{\lambda_{+}(t,x^{-}(t,\alpha))-\Lambda_{-}(\alpha)}{\Lambda_{+}(\alpha)-\Lambda_{-}(\alpha)}
 \eqno(2.10)
 $$and
 $$
 \frac{\partial\lambda_{-}(t,x)}{\partial x}\bigg|_{x=x^{-}(t,\alpha)}=\Lambda^{'}_{-}(\alpha)\frac{\Lambda_{+}(\alpha)-\Lambda_{-}(\alpha)}{\lambda_{+}(t,x^{-}(t,\alpha))-\Lambda_{-}(\alpha)}.
 \eqno(2.11)
 $$
 Similar result holds for $x=x^{+}(t,\beta)$ and $\lambda_{+}(t,x^{+}(t,\beta))$.
 \end{lemma}
\begin{remark}
It follows from $(2.11)$ that if there exists time $t_{0}$ which satisfies  $$\lambda_{+}(t_{0},x^{-}(t_{0},\alpha))=\Lambda_{-}(\alpha),\:\: \Lambda_{-}(\alpha)\neq\Lambda_{+}(\alpha)\:\: \text{and} \:\:\Lambda^{'}_{-}(\alpha)\neq0\;\text{for some}\, \alpha,$$then the solution of the Cauchy problem (2.3) with (2.5) must blow up at the time $t_{0}$. By the theory of characteristic method, we observe that $\lambda_{-}(t,x)$ and $\lambda_{+}(t,x)$ are bounded, while $(\lambda_{-})_{x}$ and $(\lambda_{+})_{x}$ tend to the infinity as $t$ goes to $t_{0}$.

 It is well known that the formation of traditional blowup, e.g., the formation of ``shock wave'' is due to the envelope of the same family of characteristics (see \cite{A,9}). However, in this paper, we shall investigate a new phenomenon on the formation of singularities which is based on the envelope of different families of characteristics (see Figure 1).
\end{remark}
\begin{figure}[h]
\centering
\includegraphics[]{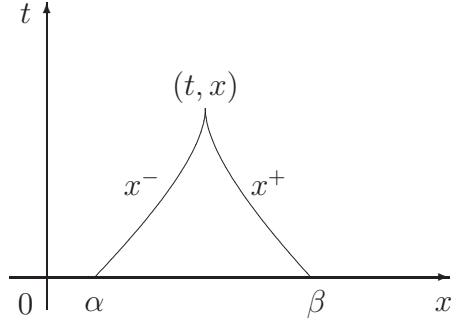}
\caption{The envelope of characteristics of different families }
\end{figure}
To do so, we suppose that the initial data $\Lambda_{-}(x)$ and $\Lambda_{+}(x)$ are suitably smooth
 functions and satisfy the following assumptions:

Assumption (H1):
$$
\Lambda_{-}(x)<\Lambda_{+}(x),\quad \forall\:  x\in \mathbb{R}.
\eqno(2.12)
$$

Assumption (H2):
$$
\Lambda^{'}_{-}(x)<0\quad \text{and}\quad \Lambda^{'}_{+}(x)<0,\qquad\forall\:x\in \mathbb{R}.
\eqno(2.13)
$$
Define
 $$\Sigma=\{(\alpha,\beta)|\alpha<\beta\;\text{and}\;\Lambda_{-}(\alpha)=\Lambda_{+}(\beta)\}.
 \eqno(2.14)$$
  In order to avoid confusion, here and hereafter, we denote the variable of $\Lambda_{-}(x)$ by $\alpha$ and the variable of $\Lambda_{+}(x)$ by $\beta$.

 By (2.13) and (2.14), for $\forall\,(\alpha,\beta)\in\Sigma$, it holds that $\beta(\alpha)=\Lambda_{+}^{-1}\Lambda_{-}(\alpha)$. Define
 $$
 f(\alpha):=\frac{\Lambda^{'}_{-}(\alpha)}{\Lambda_{+}(\beta(\alpha))-\Lambda_{-}(\beta(\alpha))}
-\frac{\Lambda^{'}_{+}(\beta(\alpha))}{\Lambda_{+}(\alpha)-\Lambda_{-}(\alpha)},
\eqno(2.15)
 $$
 where $$\Lambda^{'}_{+}(\beta(\alpha))=\frac{d\Lambda_{+}(\beta)}{d\beta}\bigg|_{\beta=\beta(\alpha)}.$$
 We furthermore assume that there exists $(\alpha_{0},\beta_{0})\in\Sigma$ such that

 Assumption (H3):
$$\Lambda_{-}(\alpha_{0})=\Lambda_{+}(\beta_{0}).
\eqno(2.16)
$$

Assumption (H4):
$$
f(\alpha_{0})=0.
\eqno(2.17)
$$

Assumption (H5):
$$
f^{'}(\alpha_{0})<0.
\eqno(2.18)
$$

For simplicity, without loss of generality, we may suppose that
$$
\Lambda_{-}(\alpha_{0})=\Lambda_{+}(\beta_{0})=0.
\eqno(2.19)
$$
This can be achieved by making a simple translation transform.

For completion of this paper, we recall the following lemmas 2.4-2.9 without proof, which can be found in \cite{KW}.
\begin{lemma}
Initial data set $\{(\Lambda_{+}(x),\Lambda_{-}(x))\}$ satisfying assumptions (H1)-(H5) is not empty.
\end{lemma}

By the existence and uniqueness theorem of a $C^{1}$ solution of the Cauchy problem for quasilinear hyperbolic systems (see \cite{5}), under the assumptions (H1)-(H5), the Cauchy problem (2.3), (2.5) has a unique $C^{1}$ solution $(\lambda_{-}(t,x),\lambda_{+}(t,x))$ in the domain $D(t_{0})\triangleq\{(t,x)|0\leq t<t_{0},-\infty<x<\infty\}$, where $t_{0}$ is just the blowup time, i.e., the life span of the $C^{1}$ solution of the Cauchy problem (2.3), (2.5). Throughout the paper, we refer $D(t_{0})$ as the existence domain of the classical solution.
\begin{lemma}
If there are two points $\alpha_{0}$ and $\beta_{0}$ satisfying (2.16), then the characteristic $x=x^{-}(t,\alpha_{0})$ must intersect $x=x^{+}(t,\beta_{0})$ in finite time, where we assume that the classical solution exists.
\end{lemma}

In what follows, under the assumptions (H1)-(H5), we first consider the Cauchy problem (2.3), (2.5).

Assume that $(\alpha_{0},\beta_{0})$ satisfies the assumptions (H1)-(H5). We introduce
$$
t_{0}=\int_{\alpha_{0}}^{\beta_{0}}\frac{1}{\Lambda_{+}(\zeta)-\Lambda_{-}(\zeta)}d\zeta
\eqno(2.20)
$$
and
$$
x_{0}=\frac{1}{2}\left\{\alpha_{0}+\beta_{0}+\int_{\alpha_{0}}^{\beta_{0}}\frac{\Lambda_{+}(\zeta)+\Lambda_{-}(\zeta)}{\Lambda_{+}(\zeta)-\Lambda_{-}(\zeta)}d\zeta\right\}.
\eqno(2.21)
$$
%
%
%

\begin{lemma}
There exists a positive constant $\epsilon$ such that $\alpha_{0}$ is the unique zero point of $f(\alpha)$, namely, $$f(\alpha_{0})=0\:\: \rm{but}\:\:f(\alpha)\neq0\:\:\quad\rm{for}\:\:\alpha\in(\alpha_{0}-\epsilon,\alpha_{0}+\epsilon).$$
\end{lemma}

It is obvious that (2.8)-(2.9) define a mapping from the region $U\triangleq\{(\alpha,\beta)\mid\alpha\leq\beta\}$ to the domain $\{(t,x)\mid t\geq0,x\in \mathbb{R}\}$. Denote it by $\Pi:$
$$
\Pi(\alpha,\beta)=(t(\alpha,\beta),x(\alpha,\beta)).
\eqno(2.22)
$$
We introduce the Jacobian matrix of $\Pi$
$$
\mathbf{\bigtriangleup(\alpha,\beta)}=
\left(\begin{array} {cc}
t_{\alpha}&t_{\beta}\\
x_{\alpha}&x_{\beta} \end{array}
\right)
\eqno(2.23)
$$\\
and its Jacobian
$$
J(\alpha,\beta) = t_{\alpha}x_{\beta} - t_{\beta}x_{\alpha}.
\eqno(2.24)
$$
Now, we can state the following definitions, which can be found in \cite{4}.
\begin{definition}
A point \textbf{p} in $U$ is called a regular point of the mapping $\Pi$
if the rank $\vartriangle$ is 2 at \textbf{p}. Otherwise, \textbf{p} is called a singular point of $\Pi$.
\end{definition}
It is easy to verify that $\textbf{p}$ is a singular point is equivalent to $\Lambda_{-}(\alpha)=\Lambda_{+}(\beta)$, which can form a smooth curve defined by an explicit function $\beta=\beta(\alpha)$, since $\Lambda_{+}^{'}(\beta)<0$.
\begin{definition}
Let \textbf{p} be a singular point of $\Pi$ and $\Upsilon(\alpha) = (\alpha,\beta(\alpha))$ be the parametric equation with $\Upsilon(\alpha_{0})=\textbf{p}$ for $J(\alpha,\beta)=0$. \textbf{p} is called a fold point of $\Pi$, if $\frac{d}{d\alpha} (\Pi\circ\Upsilon)(\alpha_{0}) \neq(0,0)$, and \textbf{p} is called a cusp point of $\Pi$, if $\frac{d}{d\alpha} (\Pi\circ\Upsilon)(\alpha_{0}) = (0,0)$ but $\frac{d^{2}}{d\alpha^{2}}(\Pi\circ\Upsilon)(\alpha_{0}) \neq (0,0).$
\end{definition}
 \begin{lemma}
 (A) The curve $\beta=\beta(\alpha)$ is strictly increasing as a function of $\alpha$;
 (B) the singular points $(\alpha,\beta)\neq(\alpha_{0},\beta_{0})$ are fold points, while $(\alpha_{0},\beta_{0})$ is a cusp point.
 \end{lemma}
\begin{lemma}
 Under the assumptions (H1)-(H5), $t_{0}$ is the unique minimum point on the interval $(\alpha_{0}-\epsilon,\alpha_{0}+\epsilon)$, where $\epsilon$ is given in Lemma 2.6.
\end{lemma}
We next discuss the position and property of $\Pi(\alpha,\beta(\alpha))$ in the $(t,x)$-plane.

Introduce $\Upsilon_{l}$ as the graph of the curve $\beta=\beta(\alpha)$ with domain $(\alpha_{0}-\epsilon,\alpha_{0})$ and $\Upsilon_{r}$ as the graph of the curve $\beta=\beta(\alpha)$ with domain $(\alpha_{0},\alpha_{0}+\epsilon)$. Then we define $\Gamma_{l}=\Pi(\Upsilon_{l})\;\text{and}\;\Gamma_{r}=\Pi(\Upsilon_{r})$.
We have the following lemma.
\begin{lemma}Under the assumptions (H1)-(H5), $\Gamma_{l}$ and $\Gamma_{r}$ form a smooth curve in (t,x)-plane which can be defined by an explicit function $t=t(x)$, moreover, $\Gamma_{l}$ is increasing and concave with respect to $x$, $\Gamma_{r}$ is decreasing and concave with respect to $x$.
\end{lemma}
Based on the properties derived in Lemmas 2.7-2.9, we can sketch the map from $(\alpha,\beta)$ to $(t,x)$ (see Figure 2).
\begin{figure}[h]
\centering
\includegraphics[trim=0 0 0 0,scale=0.80]{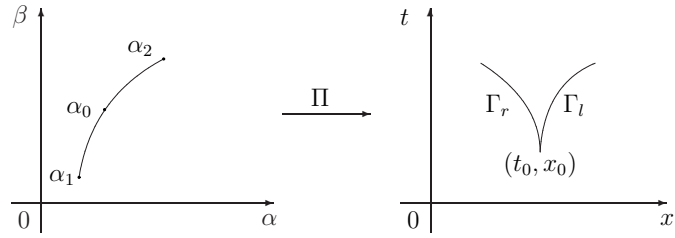}
\caption{ The mapping $\Pi$ under the assumptions (H1)-(H5). }
\end{figure}

\begin{remark}
 Passing through the point $(t_{0},x_{0})$, there exist only two characteristics which intersect the $x$-axis at $\alpha_{0}$ and $\beta_{0}$, respectively. At the point $(t_{0},x_{0})$, it holds that $$\frac{dx}{dt}=\Lambda_{-}(\alpha_{0})=\Lambda_{+}(\beta_{0})=0,$$ that is to say, these two characteristics are tangent at $(t_{0},x_{0})$.
\end{remark}
Lemmas 2.1, 2.3, 2.7 and 2.8 lead to the following main result.
\begin{theorem}
Under assumptions (H1)-(H5), the smooth solution of Cauchy problem (2.3) and (2.5) blows up at $(t_{0},x_{0})$ which is defined by (2.20)-(2.21), and $t_{0}$ is the blowup time. Furthermore, the blowup is geometric blowup.
\end{theorem}
\begin{remark}
Kong \cite{9} presented a similar discussion for a system with genuinely nonlinear characteristic fields.
\end{remark}
 The following lemma states the main phenomena of the concentration of the density $\rho(t,x)$ in the system (1.3).
 \begin{lemma}
 $\Gamma_{l}$ and $\Gamma_{r}$ are two envelopes, $(t_{0},x_{0})$, defined by (2.20) and (2.21) is the blowup point. Moreover, $\rho(t,x)$ tends to the infinity as $(t,x)$ goes to $(t_{0},x_{0})$ with $t\leq t_{0}$.
 \end{lemma}
\begin{proof}
The proof can be done by a straight forward calculation.

By (1.4) and (2.2), we have
$$
u(t,x)=\frac{\lambda_{+}(t,x)+\lambda_{-}(t,x)}{2},\qquad \rho(t,x)=\frac{2\mu}{\lambda_{+}(t,x)-\lambda_{-}(t,x)}.
$$
Thus, by the definition of $\Gamma_{l}$, $\Gamma_{r}$ and Lemma 2.8, the lemma is proved directly.
\end{proof}
\begin{remark}
Lemma 2.10 states an important fact: the starting point of the concentration of the state variable $\rho(t,x)$ in system (1.3) is the first blowup point, at which the two characteristics tangent to each other. We can conjecture that the mechanism of the formation of $\delta$-shock wave, which will be constructed in the following, is due to the degeneracy of strict hyperbolicity. Namely, at those points, $\delta$-shock wave forms.
\end{remark}
 The following lemma plays an important role in the discussion of next section.
\begin{lemma}
The solution $(\rho,u)$ of the Cauchy problem (1.3), (2.4) satisfies (1.2) for all $\phi\in C_{0}^{1}((0,t_{0}]\times\mathbb R)$, where the vector function $u$ in (1.2) stands for $(\rho,u)$.
\end{lemma}
\begin{cor}
By Lemma 2.11, it is easy to see that the following Rankine-Hugonoit conditions for shock hold at $t=t_{0}$.
$$
\frac{dx(t_{0})}{dt}[\rho]=[\rho u]
\eqno(2.56)
$$
and
$$
\frac{dx(t_{0})}{dt}[\rho u]=[\rho u^{2}+p],
\eqno(2.57)
$$
where $\frac{dx(t_{0})}{dt}=\frac{dx(t)}{dt}|_{t=t_{0}}$ and $(t_{0},x(t_{0}))$ is the starting point of the discontinuity of the curve $x(t)$.
\end{cor}
\section{Construction and uniqueness of $\delta$-shock wave solution}
In this section, we will construct the $\delta$-shock wave solution after the blowup time
$t_{0}$ defined by (2.20) and prove the uniqueness of $\delta$-shock wave solution.

In what follows, we denote the cusp point by $O$, where the characteristics of different families are tangent. If the characteristic lines are regarded as the paths of free particles, then they stick at $O$ and form a massive particle in a domain . The concentration of free particles means that the density function becomes a Dirac measure. Thus, the trajectory of this massive particle is just the $\delta$-shock to be discussed below. Mathematically, the formation of $\delta$-shock may result from the overlap of linearly degenerate characteristics, which is just as the formation of the classical shock due to the overlap of genuinely nonlinear characteristic. See Figure 3.
\begin{figure}[h]
\centering
\includegraphics[]{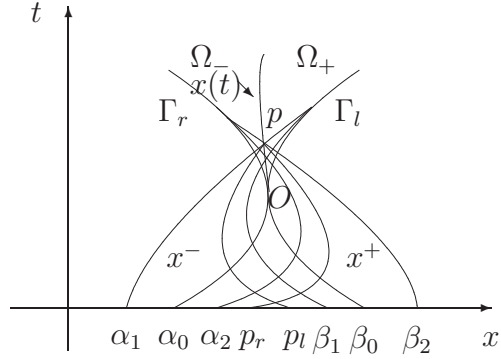}
 \caption{ The overlap of characteristics}
\end{figure}

The definition of $\delta$-measure solution in Section 1 enables us to discuss the piecewise smooth solution of (1.3). Let a smooth discontinuity $\mathbb{S}={(t,x(t))}$, on which $\rho$ becomes a Dirac measure, divide the neighborhood of $O$ into two regions $\Omega_{-}$ and $\Omega_{+}$, see Figure 3. Consider the solution of the form
$$
(\rho,u)(t,x)=\left\{
\begin{array}{l}
(\rho_{-},u_{-})(t,x),\qquad \forall(t,x)\in\Omega_{-},\vspace{2mm}\\
(w\delta_{\mathbb{S}},u_{\delta})(t,x),\quad\;\; \forall(t,x)\in \mathbb{S},\vspace{2mm}\\
(\rho_{+},u_{+})(t,x),\qquad \forall(t,x)\in\Omega_{+},
\end{array}
\right.
\eqno(3.1)
$$
where $\delta$ is the Dirac measure with the support $\mathbb{S}$, $(\rho_{-},u_{-})$ and $(\rho_{+},u_{+})$ are the bounded smooth solution of (1.3) in $\Omega_{-}$ and $\Omega_{+}$, respectively. $\mathbb{S}$ is called a $\delta$-shock, $w(t,x)$ and $u_{\delta}(t,x)$ are its weight and propagation speed, respectively.

Before state our main result, we solve $(\rho_{-},u_{-})(t,x)$ and $(\rho_{+},u_{+})(t,x)$ first by the characteristic method.

Let $\Omega_{1}$ be the region bounded by the $x$-axis, the curve $\widehat{\beta_{0} O}\cup \Gamma_{l}$ and $x^{-}(t,\alpha_{1})$, $\Omega_{2}$ be the region bounded by the $x$-axis, the curve $\widehat{\alpha_{0}O}\cup \Gamma_{r}$ and $x^{+}(t,\beta_{2})$. Here $x^{\pm}$ are defined by (2.6) and (2.7). See Figure 3.

Let
$$
\Omega:=\Omega_{1}\cap\Omega_{2}\cap\{t\geq t_{0}\}
$$

By the method of characteristic, the initial data defined on the interval $[\alpha_{1},\beta_{0}]$ determines two solutions on the domain $\Omega$, since from any point $p\in\Omega$, there exist two slow characteristics and one fast characteristic starting from the initial data defined on the interval $[\alpha_{1},\beta_{0}]$. Moreover, the two solutions is smooth and bounded except on the curve $\Gamma_{l}$. We denote by $(\rho,u)=(\rho_{1},u_{1})(t,x):=\{(\rho_{11},u_{11}),(\rho_{12},u_{12})\}$ and find that
$$
(\rho_{1i},u_{1i})=(\rho_{1i},u_{1i})(t,x)\in C^{1}(\overline{\Omega}\backslash \Gamma_{l})\quad \text{for}\;i=1,2.
\eqno(3.2)
$$

Similarly, the initial data defined on the interval $[\alpha_{0},\beta_{2}]$ determines two solutions on the domain $\Omega$ and the two solutions are smooth and bounded except on the curve $\Gamma_{r}$. We denote the solution by $(\rho,u)=(\rho_{2},u_{2})(t,x):=\{(\rho_{21},u_{21}),(\rho_{22},u_{22})\}$ and observe that
$$
(\rho_{2i},u_{2i})=(\rho_{2i},u_{2i})(t,x)\in C^{1}(\overline{\Omega}\backslash\Gamma_{r})\quad \text{for}\;i=1,2.
\eqno(3.3)
$$

Obviously, on $\Omega$ there exist four solutions. By Figure 3, we can see that from $p\in\Omega$, there are four characteristics coming from the initial data.
\begin{remark}
By (3.2) and (3.3), there must be a gradient catastrophe in the domain $\Omega$. Thus, the aim of $\delta$-shock wave solution is to split the domain $\Omega$ into two parts $\Omega_{-}$ and $\Omega_{+}$ such that it can prevent the characteristics from outgoing. Then  one can choose uniquely $(\rho_{-},u_{-})(t,x)\in(\rho_{1},u_{1})(t,x)$ and $(\rho_{+},u_{+})(t,x)\in(\rho_{2},u_{2})(t,x)$, which satisfy system (1.3) in their corresponding domain. Moreover£¬
$$
(\rho_{-},u_{-})(t,x)\in C^{1}\cap L^{\infty}(\Omega_{-})\quad and\quad (\rho_{+},u_{+})(t,x)\in C^{1}\cap L^{\infty}(\Omega_{+}).
$$
Take the point $p$ for example, on the left of $p$, $(\rho_{-},u_{-})(t,x)$ are determined by the two characteristics $x^{-}(t,\alpha_{1})$ and $x^{+}(t,p_{l})$, while on the right of $p$, $(\rho_{+},u_{+})(t,x)$ are determined by the two characteristics $x^{-}(t,p_{r})$ and $x^{+}(t,\beta_{2})$.
\end{remark}
Now the main theorem of this section can be stated as follows. For simplicity, we assume the blowup point $O$ to be the origin of $(t,x)$-plane.
\begin{theorem}
Under the assumptions (H1)-(H5) on the initial data, There exists a constant $T>0$, such that the Cauchy problem (1.3) and (2.4) admits a unique $\delta$-shock wave type solution on $[0,T)$ and it takes the following form
$$
(\rho,u)(t,x)=\left\{
\begin{array}{l}
(\rho_{-},u_{-})(t,x)\qquad (t,x)\in\Omega_{-}\vspace{2mm}\\
(w\delta_{\mathbb{S}},u_{\delta})(t,x)\quad\;\,\, (t,x)\in \mathbb{S}\vspace{2mm}\\
(\rho_{+},u_{+})(t,x)\qquad (t,x)\in\Omega_{+},
\end{array}
\right.
$$
which satisfies the integral identities (1.8) and (1.9) in the sense of Definition 1.3, where $(\rho_{-},u_{-})(t,x)\in C^{1}\cap L^{\infty}(\Omega_{-})$ and $(\rho_{+},u_{+})(t,x)\in C^{1}\cap L^{\infty}(\Omega_{+})$ satisfy system (1.3), $w$ and $u_{\delta}$ are the weight and propagation speed with initial data $w(0)\neq0$ and $u_{\delta}(0)=0$, respectively, and they satisfy the following generalized Rankine-Hugoniot conditions
$$
\frac{d(w(t)\sqrt{1+u_{\delta}^{2}(t)})}{dt}=u_{\delta}(t)[\rho]-[\rho u]
\eqno(3.4)
$$
$$
\frac{d(w(t)u_{\delta}(t)\sqrt{1+u_{\delta}^{2}(t)})}{dt}=u_{\delta}(t)[\rho u]-[\rho u^{2}+p(\rho)],
\eqno(3.5)
$$
where $
[h]=h_{+}-h_{-}
$ denotes the jump across the discontinuity curve $x(t)$.
Furthermore, the admissibility $\delta$-entropy condition
$$
\lambda_{-}^{r}(t,x(t))<\lambda_{+}^{r}(t,x(t))\leq\frac{dx(t)}{dt}\leq\lambda_{-}^{l}(t,x(t))<\lambda_{+}^{l}(t,x(t)).
\eqno(3.6)
$$
holds in the domain $\Omega$, where $\lambda_{\pm}^{l}$ are the left limits on the discontinuity of the two characteristic fields' speeds and $\lambda_{\pm}^{r}$ are the right limits on the discontinuity of the two characteristic fields' speeds.
\end{theorem}
\begin{proof}
We first check that the constructed $\delta$-measure solution in the theorem satisfies Definition 1.3 in the sense of distributions, that is,
$$
I:\qquad<\rho,\phi_{t}>+<\rho u,\phi_{x}>=0
$$
and
$$
II:\qquad<\rho u,\phi_{t}>+<\rho u^{2}+p(\rho),\phi_{x}>=0
$$
for all test functions $\phi(x,t)\in C_{0}^{\infty}(D)$, where $D$ is a ball centered at $Q$ and $Q$ is any point on $\mathbb{S}$. Furthermore, we assume that $D$ intersect $\mathbb{S}$ at $Q_{1}=(t_{1},x(t_{1}))$ and $Q_{2}=(t_{2},x(t_{2}))$. Let $D_{1}$ and $D_{2}$ be the components of $D$ which are determined by $\mathbb{S}$. By direct calculations and Green's formula, we have
$$
\begin{array}{l}
I=<\hat{\rho},\phi_{t}>+<w(t)\delta_{\mathbb{S}},\phi_{t}>+<\hat{\rho}\hat{u},\phi_{x}>+<w(t)u_{\delta}\delta_{\mathbb{S}},\phi_{x}>\vspace{2mm}\\
\;\;\,=\iint_{D}\hat{\rho}\phi_{t}dxdt+\iint_{D}\hat{\rho}\hat{u}\phi_{x}dxdt
+\int_{t_{1}}^{t_{2}}w(t)\phi_{t}(t,x(t))\sqrt{1+\dot{x}^{2}(t)}dt+\vspace{2mm}\\
\;\;\,\quad\int_{t_{1}}^{t_{2}}w(t)u_{\delta}\phi_{x}(t,x(t))\sqrt{1+\dot{x}^{2}(t)}dt\vspace{2mm}\\
\;\;\,=\iint_{D_{1}}\rho_{-}\phi_{t}dxdt+\iint_{D_{2}}\rho_{+}\phi_{t}dxdt+\iint_{D_{1}}\rho_{-}u_{-}\phi_{x}dxdt
+\iint_{D_{2}}\rho_{+}u_{+}\phi_{x}dxdt+\vspace{2mm}\\
\;\;\,\quad\int_{t_{1}}^{t_{2}}w(t)(\phi_{t}+\frac{dx(t)}{dt}\phi_{x})(t,x(t))\sqrt{1+\dot{x}^{2}(t)}dt\vspace{2mm}\\
\;\;\,=\iint_{D_{1}}(\rho_{-}\phi)_{t}dxdt+\iint_{D_{2}}(\rho_{+}\phi)_{t}dxdt
+\iint_{D_{1}}(\rho_{-}u_{-}\phi)_{x}dxdt+\iint_{D_{2}}(\rho_{+}u_{+}\phi)_{x}dxdt-\vspace{2mm}\\
\;\;\,\quad\int_{t_{1}}^{t_{2}}\frac{d(w(t)\sqrt{1+\dot{x}^{^{2}}(t)})}{dt}\phi dt\vspace{2mm}\\
\;\;\,=\iint_{D_{1}}[(\rho_{-}\phi)_{t}+(\rho_{-}u_{-}\phi)_{x}]dxdt+\iint_{D_{2}}[(\rho_{+}\phi)_{t}+(\rho_{+}u_{+}\phi)_{x}]dxdt
-\int_{t_{1}}^{t_{2}}\frac{d(w(t)\sqrt{1+\dot{x}^{^{2}}(t)})}{dt}\phi dt\vspace{2mm}\\
\;\;\,=\oint_{\partial D_{1}}-\rho_{-}\phi dx+\rho_{-}u_{-}\phi dt+\oint_{\partial D_{2}}-\rho_{+}\phi dx+\rho_{+}u_{+}\phi dt-\int_{t_{1}}^{t_{2}}\frac{d(w(t)\sqrt{1+\dot{x}^{^{2}(t)}})}{dt}\phi dt\vspace{2mm}\\
\;\;\,=\int_{t_{1}}^{t_{2}}[(-\rho_{-}+\rho_{+})\frac{dx(t)}{dt}+(\rho_{-}u_{-}-\rho_{+}u_{+})]\phi dt-\int_{t_{1}}^{t_{2}}\frac{d(w(t)\sqrt{1+\dot{x}^{^{2}}(t)})}{dt}\phi dt\vspace{2mm}\\
\;\;\,=0,
\end{array}
$$
where $\partial D_{1,2}$ is the boundaries of $D_{1,2}$ and $\oint_{\partial D_{1,2}}$ is the line integral on boundaries $\partial D_{1,2}$ and we denote $\frac{dx}{dt}$ by $\dot{x}$ for any function $x(t)$ here and hereafter. The last equality comes from (3.4).
Similarly, we have
$$
\begin{array}{l}
II
=\iint_{D}\hat{\rho}\hat{u}\phi_{t}dxdt+\int_{t_{1}}^{t_{2}}w(t)u_{\delta}\phi_{t}(t,x(t))\sqrt{1+\dot{x}^{2}}dt
+\iint_{D}(\hat{\rho}\hat{u}^{2}+\hat{p})\phi_{x}dxdt+\vspace{2mm}\\
\qquad\,\,\int_{t_{1}}^{t_{2}}w(t)u_{\delta}^{2}\phi_{x}(t,x(t))\sqrt{1+\dot{x}^{2}}dt\vspace{2mm}\\
\quad\;=\iint_{D_{1}}(\rho_{-}u_{-}\phi)_{t}dxdt+\iint_{D_{2}}(\rho_{+}u_{+}\phi)_{t}dxdt+
\iint_{D_{1}}\left((\rho_{-}u_{-}^{2}+p(\rho_{-}))\phi\right)_{x}dxdt+\vspace{2mm}\\
\qquad\,\iint_{D_{2}}\left((\rho_{+}u_{+}^{2}+p(\rho_{+}))\phi\right)_{x}dxdt+\int_{t_{1}}^{t_{2}}w(t)u_{\delta}(\phi_{t}+\dot{x}(t)\phi_{x})\sqrt{1+\dot{x}(t)^{2}}dt\vspace{2mm}\\
\quad\;=\int_{t_{1}}^{t_{2}}\left(u_{\delta}(t)[\rho u]-[\rho u^{2}+p(\rho)]\right)\phi dt-\int_{t_{1}}^{t_{2}}\frac{d(w(t)u_{\delta}(t)\sqrt{1+u_{\delta}^{2}(t)})}{dt}\phi dt\vspace{2mm}\\
\quad\;=0.
\end{array}
$$
The last equality is due to (3.5) and $\dot{x}(t)=u_{\delta}(t)$.

For the existence of the $\delta$-shock wave solution, we can solve the system of ODEs (3.4) and (3.5) with their corresponding initial data. By (3.4) and (3.5), we have
$$
\sqrt{1+u_{\delta}^{2}(t)}\dot{w}(t)+\frac{w(t)u_{\delta}(t)\dot{u}_{\delta}(t)}{\sqrt{1+u^{2}_{\delta}(t)}}=u_{\delta}(t)[\rho]-[\rho u]
\eqno(3.7)
$$
and
$$
\dot{w}(t)u_{\delta}(t)\sqrt{1+u_{\delta}^{2}(t)}+w(t)\sqrt{1+u_{\delta}^{2}(t)}\dot{u}_{\delta}(t)+
\frac{w(t)u_{\delta}^{2}(t)\dot{u}_{\delta}(t)}{\sqrt{1+u_{\delta}^{2}(t)}}=u_{\delta}(t)[\rho u]-[\rho u^{2}+p(\rho)].
\eqno(3.8)
$$
By direct calculations, the following system holds
$$
\ddot{x}(t)=\frac{\dot{x}(t)[\rho u]-[\rho u^{2}+p(\rho)]-\dot{x}(t)(\dot{x}(t)[\rho]-[\rho u])}{w(t)\sqrt{1+\dot{x}^{2}(t)}}:=F(t,x(t),\dot{x}(t),w(t))
\eqno(3.9)
$$
and
$$
\dot{w}(t)=-\frac{\dot{x}(t)\left(\dot{x}(t)[\rho u]-[\rho u^{2}+p(\rho)]-\dot{x}(t)(\dot{x}(t)[\rho]-[\rho u])\right)}{(1+\dot{x}^{2}(t))^{\frac{3}{2}}}+\frac{\dot{x}(t)[\rho]-[\rho u]}{\sqrt{1+\dot{x}^{2}(t)}}.
\eqno(3.10)
$$
Then by the existence and uniqueness theorem of the initial value problem of ODEs, we know that there exists a constant $T_{1}>0$ such that the solution of (3.9) and (3.10) exist on $[0,T_{1})$. Since by the above construction $(\rho_{-},u_{-})$ and $(\rho_{+},u_{+})$ are in $C^{1}$, $w(0)\neq0$ and on $t=0$, by Corollary 3.1, we have
$$
\dot{x}(t)[\rho]-[\rho u]=\dot{x}(t)[\rho u]-[\rho u^{2}-p(\rho)]\equiv0.
$$

Next, we show that there exists a constant $T_{2}>0$ such that the discontinuity curve $x(t)$ lies between $\Gamma_{r}$ and $\Gamma_{l}$ under the initial data $x(0)=\dot{x}(0)=0$ and $w(0)\neq0$ for $t\in[0,T_{2})$.
By (3.7), (3.8) and the initial condition $w(0)\neq0\; \text{and}\; u_{\delta}(0)=0$,
we have
$$
\frac{d^{2}x(t)}{dt^{2}}\bigg|_{t=0}=\dot{u}_{\delta}(t)|_{t=0}=0.
\eqno(3.11)
$$
Since on $t=0$, it holds that
$$
u_{\delta}(t)[\rho]-[\rho u]=u_{\delta}(t)[\rho u]-[\rho u^{2}+p(\rho)]=0.
$$
By Lemma 2.9 or Figure 2, we have
$$
\frac{d^{2}\Gamma_{r}}{dt^{2}}\bigg|_{t=0}<0\quad \text{and} \quad \frac{d^{2}\Gamma_{l}}{dt^{2}}\bigg|_{t=0}>0.
\eqno(3.12)
$$
By the initial condition $x(0)=\dot{x}(0)=0$ and (3.11)-(3.12), there exists a small constant $T_{2}$ such that for $t\in[0,T_{2})$, it holds that
$$
\frac{d\Gamma_{r}}{dt}<\frac{dx(t)}{dt}<\frac{d\Gamma_{l}}{dt}.
\eqno(3.13)
$$
Therefore, we have
$$
\Gamma_{r}<x(t)<\Gamma_{l}\quad \rm{for}  \; t\in(0,T_{2}).
\eqno(3.14)
$$
Then by the method of characteristic and (3.13), it is easy to show that
$$
\lambda_{-}^{r}(t,x(t))<\lambda_{+}^{r}(t,x(t))\leq\frac{d\Gamma_{r}}{dt}\leq\frac{dx(t)}{dt}\leq\frac{d\Gamma_{l}}{dt}\leq\lambda_{-}^{l}(t,x(t))<\lambda_{+}^{l}(t,x(t)),
\eqno(3.15)
$$
for $t\in[0,T_{2})$,
which means that all characteristics on both sides of the discontinuity are not outcoming and compatible with $\delta$-entropy condition, which also guarantees the uniqueness of $\delta$-shock wave solution.

Finally, take $T=min\{T_{1},T_{2}\}$, then the theorem holds.
\end{proof}
\begin{definition}
A discontinuity which is presented in the form (3.1) and satisfies (3.4)-(3.6) will be called a delta shock wave solution to system (1.3), symbolized by $\delta$.
\end{definition}
\begin{remark}
The initial data of the weight $w(0)\neq0$ and the initial speed of the discontinuity $\dot{x}(0)=0$ are reasonable. $w(0)\neq0$ and $\dot{x}(0)=0$ are sufficient to keep the discontinuity $x(t)$ lies between $\Gamma_{l}$ and $\Gamma_{r}$ and make sure that the constructed solution satisfies the admissible $\delta$-entropy condition.
\end{remark}
\begin{remark}
Different from other papers on the $\delta$-shock wave solution of hyperbolic conservation laws, we have to deal with a weaker case. In fact, by the above construction of (3.1), at the blowup point $t_{0}$, $\rho_{\pm}=\infty$, which may take some trouble to the uniqueness of the generalized entropy condition. Fortunately, our analysis of the behavior of singularity (see \cite{KW}) helps us avoiding this difficulty.
\end{remark}
\begin{remark}
The Delta-like singularity with point-shape and Delta-like singularity with line-shape in our paper \cite{KW} satisfy the generalized entropy condition in Theorem 3.1.
\end{remark}
\section{Physical meaning}
In this section, we will show the physical meaning of $\delta$-shock wave solution constructed in last section.

Since the density $\rho(t,x)$ contains the Dirac delta function, classical conservation laws do not hold. However, there are generalized conservation laws for system (1.3).

Define
$$
S_{\rho}(t)=\int_{-\infty}^{\infty}\rho(x,t)dx=\int_{-\infty}^{x(t)}\rho(x,t)dx+\int_{x(t)}^{\infty}\rho(x,t)dx
\eqno(4.1)
$$
and
$$
S_{\rho u}(t)=\int_{-\infty}^{\infty}\rho u(x,t)dx=\int_{-\infty}^{x(t)}\rho u(x,t)dx+\int_{x(t)}^{\infty}\rho u(x,t)dx.
\eqno(4.2)
$$
Where $(\rho(x,t),u(x,t))$ are defined by (4.1), $x(t)$ is the discontinuity, which is the support of the $\delta$-shock. Then, we have the following total mass and momentum conservation.
\begin{theorem}
Let $(\rho,u)=(\rho(t,x),u(t,x))$ defined in last section be the $\delta$-shock wave type solution of the Cauchy problem (1.3), (2.4) and $x(t)$ is the discontinuity curve, suppose that $(\rho,u)$ is compactly supported by a constant state (vacuum state is not included here) at infinity with respect to $x$. Then the following balance laws hold:
$$
\dot{S}_{\rho}(t)=-\frac{d\left(w(t)\sqrt{1+u_{\delta}^{2}(t)}\right)}{dt}=-u_{\delta}(t)[\rho]+[\rho u],
\eqno(4.3)
$$
and
$$
\dot{S}_{\rho u}(t)=-\frac{d\left(w(t)u_{\delta}(t)\sqrt{1+u_{\delta}^{2}(t)}\right)}{dt}=-u_{\delta}(t)[\rho u]+[\rho u^{2}+p(\rho)],
\eqno(4.4)
$$
namely
$S_{\rho}(t)+w(t)\sqrt{1+u_{\delta}^{2}(t)}$ and $S_{\rho u}(t)+w(t)u_{\delta}(t)\sqrt{1+u_{\delta}^{2}(t)}$ are conserved. We call them generalized mass and momentum conservation law respectively.
\end{theorem}
\begin{proof}
Using the system (1.3) and differentiating (4.1) with respect to $t$ gives
$$
\begin{array}{l}
\dot{S}_{\rho}(t)=\rho_{-}(x,t)\dot{x}(t)-\rho_{+}(x,t)\dot{x}(t)+\int_{-\infty}^{x(t)}\rho_{t}dx+\int_{x(t)}^{\infty}\rho_{t}dx\vspace{2mm}\\
\qquad\;\,=-\dot{x}(t)[\rho]-\int_{-\infty}^{x(t)}(\rho u)_{x}dx-\int_{x(t)}^{\infty}(\rho u)_{x}dx\vspace{2mm}\\
\qquad\;\,=-\dot{x}(t)[\rho]+[\rho u]+\rho u(-\infty,t)-\rho u(\infty,t)\vspace{2mm}\\
\qquad\;\,=-\dot{x}(t)[\rho]+[\rho u].
\end{array}
$$
The last equality holds since $\rho u(-\infty,t)=\rho u(\infty,t)$ by the compactness of the assumption in the theorem. Thus, the equality (4.3) holds.

(4.4) follows similarly.
\end{proof}
\begin{remark}
In contrast to the classical shock, the $\delta$-shock carries something like mass, momentum or energy and so on.
\end{remark}
\section{some remarks and discussions}
In this paper, we investigate the cusp-type singularity of a linearly degenerate and nonstrict hyperbolic conservation laws under some special assumptions on the initial data. Since singularities in hyperbolic conservation laws are complex and interesting in this research topic, it is necessary to point out that what singularity can form. Hoppe \cite{3} studied the Born-Infeld equation and derived a swallowtail singularity by the self-similar method. Thus, it is important to investigate this phenomenon in mathematical physics especially in the string theory.

Different from the Riemann problem, it requires more careful analysis in the Cauchy problem, especially the behavior of the solution in the neighborhood of the blowup point. For the construction of the nonclassical solution of system (1.3), we can see that the behavior of the neighborhood of the blowup point plays an important role in the study on uniqueness of $\delta$-shock wave solutions.

 Finally, we would like to point out that the solution $(\rho_{-},u_{-})$ and $(\rho_{+},u_{+})$ constructed in the form (3.1) are not bounded everywhere, which in some sense generalizes the definition of $\delta$-shock wave solution.

\vskip 4mm

\noindent{\Large {\bf Acknowledgements.}} The authors thank D. Christodoulou and T. Luo for helpful discussions.
This work was supported in
part by the NNSF of China (Grant No.: 11271323), Zhejiang Provincial Natural Science Foundation of China (Grant No.: Z13A010002) and a National Science and Technology Project during the twelfth five-year plan of China (2012BAI10B04).

\end{document}